\documentclass[11pt]{amsart}

\usepackage{amssymb}
\usepackage{mathrsfs}
\usepackage{amsfonts}
\usepackage{latexsym,amsmath,amsthm,amssymb,amsfonts}
\usepackage[usenames]{color}
\usepackage{xspace,colortbl}
\usepackage{graphicx}
\usepackage{tipa}

\newcommand{\be}{\begin{equation}}
\newcommand{\ee}{\end{equation}}
\newcommand{\beq}{\begin{eqnarray}}
\newcommand{\eeq}{\end{eqnarray}}

\newtheorem{prop}{Proposition}[section]

\newtheorem{remark}[prop]{Remark}

\def\begeq{\begin{equation}}
\def\endeq{\end{equation}}

\def\odot{\setbox0=\hbox{$\bigcirc$}\relax \mathbin {\hbox
to0pt{\raise.5pt\hbox to\wd0{\hfil $\wedge$\hfil}\hss}\box0 }}

\numberwithin{equation} {section}

\numberwithin{equation}{section}
\newtheorem{theorem}{\bf Theorem}[section]

\newtheorem{lemma}[theorem]{\bf Lemma}

\newtheorem{corollary}[theorem]{\bf Corollary}

\allowdisplaybreaks

\begin{document}

\title[translating surfaces of the non-parametric mean curvature flow in $M^{2}\times\mathbb{R}$]
 {translating surfaces of the non-parametric mean curvature flow in Lorentz manifold $M^{2}\times\mathbb{R}$}

\author{
Li Chen, Dan-Dan Hu, Jing Mao$^{\ast}$ and Ni Xiang}
\address{
Faculty of Mathematics and Statistics, Key Laboratory of Applied
Mathematics of Hubei Province, Hubei University, Wuhan 430062,
China. } \email{ jiner120@163.com}

\thanks{
$\ast$ Corresponding author}

\date{}

\begin{abstract}
In this paper, for the Lorentz manifold $M^{2}\times\mathbb{R}$,
with $M^{2}$ a $2$-dimensional complete surface with nonnegative
Gaussian curvature, we investigate its space-like graphs over
compact strictly convex domains in $M^{2}$, which are evolving by
the non-parametric mean curvature flow with prescribed contact angle
boundary condition, and show that solutions converge to ones moving
only by translation.
\end{abstract}

\maketitle {\it \small{{\bf Keywords}: Translating surfaces, mean
curvature flow, Lorentz manifolds.}

 }

\section{Introduction}

In Riemmnain (or pseudo-Riemannian) geometry, the mean curvature
flow (MCF for short) is actually evolving a family of immersed
submanifolds along their mean curvature vectors $\vec{H}$ with a
speed $|\vec{H}|$. More precisely, let $X_{0}:M^{n}\rightarrow
N^{m+n}$ be an isometric immersion from an $n$-dimensional oriented
Riemannian manifold $M$ to an $(n+m)$-dimensional Riemannian (or
pseudo-Riemannian) manifold $N^{n+m}$ (or with a pseudo-Riemannian
metric whose signature is $(p,n+m-p)$, $n\leq p\leq n+m-1$). The MCF
corresponds to a one-parameter family $X(\cdot,t)=X_{t}$ of
immersions $X_{t}:M^{n}\rightarrow N^{n+m}$ whose images
$M_{t}^{n}=X_{t}(M^{n})$ satisfy
\begin{eqnarray} \label{MCF}
\left\{
\begin{array}{lll}
 \frac{d}{dt}X(x,t)=\vec{H},\qquad \qquad &\mathrm{on} ~ M^{n}\times[0,\mathbb{T})\\
 X(x,0)=X_{0}(x),\qquad \qquad &\mathrm{on} ~M^{n},
 \end{array}
\right.
 \end{eqnarray}
for some $\mathbb{T}>0$. The MCF attracts a lot of attention since
Huisken's significant work \cite{gh1}, where, by using the method of
$L^{p}$ estimates, he proved that if $M^{n}$ is a compact strictly
convex hypersurface in the Euclidean $(n+1)$-space
$\mathbb{R}^{n+1}$, the MCF (\ref{MCF}) has a unique smooth solution
on the finite time interval $[0,\mathbb{T}_{\max})$ with
$\mathbb{T}_{\max}<\infty$, and the evolving hypersurfaces
$M^{n}_{t}$ contract to a single point as
$t\rightarrow\mathbb{T}_{\max}$. Moreover, after an area-preserving
rescaling, the rescaled hypersurfaces converge in
$C^{\infty}$-topology to a round sphere having the same area as
$M^{n}$. For the MCF (\ref{MCF}), if
\begin{eqnarray*}
\vec{H}=-X^{\perp},
\end{eqnarray*}
then the submanifold $X_{t}:M^{n}\rightarrow N^{n+m}$ is called a
\emph{self-shrinker}, which is a self-similar solution to
(\ref{MCF}).
 Here $(\cdot)^{\perp}$ denotes the normal
projection of a prescribed vector to the normal bundle of
$M_{t}^{n}$ in $N^{m+n}$. Self-shrinking solutions are important in
the study of type-I singularities of MCF. For instance, by proving
the monotonicity formula, at a given type-I singularity of the MCF,
Huisken \cite{gh2} proved that the flow is asymptotically
self-similar, which implies that in this situation the flow can be
modeled by self-shrinking solutions. If there exists a constant unit
vector $V$ such that
\begin{eqnarray*}
\vec{H}=V^{\perp},
\end{eqnarray*}
then the submanifold $X_{t}:M^{n}\rightarrow N^{n+m}$ is called a
\emph{translating soliton} of the MCF (\ref{MCF}). It is easy to see
that the translating soliton gives an eternal solution
$X_{t}=X_{0}+tV$ to (\ref{MCF}), which is called the translating
solution. Translating solitons play an important role in the study
of type-II singularities of the MCF. For instance, Angenent and
Velazquez \cite{av1,av2} gave some examples of convergence which
implies that type-II singularities of the MCF there are modeled by
translating surfaces.

From the above brief introduction, we know that translating
solutions of the MCF are special solutions to the flow equation and
are worthy of be investigated for the purpose of understanding
type-II singularities of the MCF. There exist some interesting
results which we prefer to mention here. For instance, Shahriyari
\cite{ls} proved that for the MCF, there are only three types of
complete translating graphs in $\mathbb{R}^{3}$, i.e., entire
graphs, graphs between two parallel planes, and graphs in one side
of a plane. Moreover, in the last two types, graphs are asymptotic
to planes next to their boundaries. For the case that
$N^{n+m}=\mathbb{R}^{n+m}$, Xin \cite{ylx} proved that any complete
translating soliton has infinite volume and has Euclidean volume
growth at least. Moreover, he showed that graphic translating
soliton hypersurfaces are weighted area-minimizing. Huisken
\cite{gh3} investigated graphs  over bounded domains (with
$C^{2,\alpha}$ boundary) in $\mathbb{R}^{n}$ ($n\geq2$), which are
evolving by the MCF with vertical contact angle boundary condition,
and proved that the evolution exists for all the time and the
evolving graphs converge to a constant function as time tends to
infinity (i.e., $t\rightarrow\infty$). Altschuler and Wu \cite{aw}
proved that graphs, defined over strictly convex compact domains in
$\mathbb{R}^{2}$, evolved by the nonparametric MCF with prescribed
contact angle (not necessary to be vertical), converge to
translating surfaces as $t\rightarrow\infty$. Guan \cite{bg}
investigated graphs over bounded domains in $\mathbb{R}^{n}$, which
are evolving by the nonparametric MCF with prescribed contact angle,
and proved that the flow exists for all the time. But an extra
assumption about the prescribed contact angle should be added in
order to get the asymptotical behavior of limiting solutions. Zhou
\cite{hyz} has improved Altschuler-Wu's and Guan's conclusions to
the case of general product spaces $M^{n}\times\mathbb{R}$ with
closed manifold $M^{n}$ having nonnegative Ricci curvature.

The purpose of this paper is to investigate the case of space-like
graphs evolved by the nonparametric MCF with the prescribed contact
angle boundary condition, and try to get interesting convergence
conclusions.

Throughout this paper, $M^{2}$ denotes a $2$-dimensional complete
Riemannian manifold with a metric $\sigma$ and $\Omega$ is a
strictly convex, bounded domain of $M^{2}$ with smooth boundary
$\partial\Omega$. Let $\kappa>0$ be the curvature function of
$\partial\Omega$. Assume that a point on $\Omega$ is described by
local coordinates $\{\omega^{1},\omega^{2}\}$. Let $\partial_{i}$,
$i=1,2$, be the corresponding coordinate vector fields and
$\sigma_{ij}=\sigma(\partial_{i},\partial_{j})$. Similar to the
basic introduction of geometry of graphs shown in \cite{cm,cmxx}, we
know that for the space-like graph
$\mathscr{G}:=\{(x,u(x,\cdot))|x\in\Omega\}$, defined over
$\Omega\subset M^{2}$, in the Lorentz manifold
$M^{2}\times\mathbb{R}$ with the metric $g:=\sigma_{ij}dw^{i}\otimes
dw^{j}-ds\otimes ds$, tangent vectors are given by
\begin{eqnarray*}
\vec{e}_{i}=\partial_{i}+D_{i}u\partial_{s}, \qquad i=1,2,
\end{eqnarray*}
and the corresponding upward unit normal vector is given by
\begin{eqnarray*}
\vec{\gamma}=\frac{1}{\sqrt{1-|Du|^2}}\left(\partial_s+D^{j}
u\partial_{j}\right),
\end{eqnarray*}
where $D^{j}u=\sigma^{ij}D_{i}u$ with $D$ the covariant derivative
operator on $M^{2}$. Denote by $\nabla$ the gradient operator on
$\mathscr{G}$, and then the second fundamental form
$h_{ij}dw^{i}\otimes dw^{j}$ of $\mathscr{G}$ is given by
\begin{eqnarray*}
h_{ij}=-\langle\nabla_{\vec{e}_i}
\vec{e}_j,\vec{\gamma}\rangle=\frac{1}{\sqrt{1-|Du|^2}}D_{i}D_{j}u.
\end{eqnarray*}
Moreover, the scalar mean curvature of $\mathscr{G}$ is
\begin{eqnarray} \label{1.2}
\qquad
H=\sum_{i=1}^{2}h^i_i=\frac{1}{\sqrt{1-|Du|^2}}\left(\sum_{i,k=1}^{2}g^{ik}D_{k}D_{i}u\right)=
\frac{\sum_{i,k=1}^{2}\left(\sigma^{ik}+\frac{D^{i}uD^{k}u}{1-|Du|^{2}}\right)D_{k}D_{i}u}{\sqrt{1-|Du|^2}}.
\end{eqnarray}

Let $T$ be the counterclockwise unit smooth tangent vector of
$\partial\Omega$ and $N$ be the inward unit normal vector of
$\partial\Omega$. Then one can smoothly extend $N$, $T$ to a thin
neighborhood of the boundary $\partial\Omega$ (see Subsection 2.1
for details).

In order to bring convenience to calculations in the sequel and
state our main conclusion clearly, we use the following notations.
$$v=\sqrt {1-|Du|^{2}},$$
$$g_{ij}=\sigma_{ij}-D_{i}uD_{j}u,$$
$$g^{ij}=\sigma^{ij}+\frac{D^{i}uD^{j}u}{1-|Du|^{2}},$$
$$u_{t}=\frac{\partial u}{\partial t}.$$
For vectors, $V$, $W$ or matrices $A$, $B$, we will use the
shorthand as follows
 \begin{eqnarray*}
 \langle V,W \rangle
_{g}=g^{ij}V_{i}W_{j};\quad \langle V,W\rangle
_{\sigma}=\sigma^{ij}V_{i}W_{j};\quad \langle A,B\rangle
_{g,\sigma}=g^{ij}\sigma^{kl}A_{ik}B_{jl}.
 \end{eqnarray*}
Define $g^{TN}:=g^{ij}T_{i}N_{j}$, $g^{TT}:=g^{ij}T_{i}N_{j}$, and
$g^{NN}:=g^{ij}N_{i}N_{j}$ on $\partial \Omega$. For the
second-order covariant derivatives of a prescribed function, we have
the formula
\begin{eqnarray*}
D_{V}D_{W}u=V^{i}W^{j}D_{ij}^{2}u+\langle D_{V}W,Du\rangle.
\end{eqnarray*}

For the space-like graphs $\mathscr{G}$, consider the following
initial value problem (IVP for short)
\begin{eqnarray*}
(\sharp)\qquad \left\{
\begin{array}{lll}
{u_{t}=\left(\sigma^{ij}+\frac{D^{i}uD^{j}u}{1-|Du|^{2}}\right)D_{i}D_{j}u}, \qquad \qquad &\mathrm{on} ~ \Omega\times[0,\mathbb{T}]\\
D_{N}u=\phi(x)v, \qquad \qquad &\mathrm{on} ~
\partial \Omega\times[0,\mathbb{T}]\\
u(\cdot,0)=u_{0}(\cdot),\qquad \qquad &\mathrm{on} ~\Omega_{0}
\end{array}
\right.
\end{eqnarray*}
where $\Omega_{t}=\Omega\times\{t\}$ is a slice in
$\Omega\times[0,\mathbb{T}]$, $\phi\in C^{\infty}(\partial\Omega)$,
and $u_{0}\in C^{\infty}(\overline{\Omega})$. Clearly, the IVP
$(\sharp)$ describes the evolution of space-like graphs
$\mathscr{G}$ by the mean curvature vector with the specified
contact angle, since by (\ref{1.2}) the RHS of the first evolution
equation in $(\sharp)$ equals $Hv$. For the IVP $(\sharp)$, we can
prove the following.

\begin{theorem}\label{main1.1}
If $\Omega$ is a strictly convex bounded domain in $M^{2}$ with
nonnegative Gaussian curvature,
then, for solutions to IVP $(\sharp)$, we have the followings:\\
 1. there exists some constant $c_{1}:= c_{1}(u_{0},\kappa_{0},\phi_{0},\phi_{1},\phi_{2})>0$ so that $|Du|^{2}\leq c_{1}<1$ on $\overline{\Omega}\times[0,\infty)$, thus $u(x,t)\in C^{\infty}(\overline{\Omega}\times[0,\infty))$, where
\begin{eqnarray*}
 \kappa_0:=\min\limits_{x\in\partial\Omega}\kappa(x), \quad
 \phi_{0}:=\min\limits_{x\in\partial\Omega}\phi(x), \quad
\phi_{1}:=\max\limits_{x\in\partial\Omega}\phi(x), \quad
\phi_{2}:=\max\limits_{x\in\partial\Omega}|D_{T}\phi(x)|;
 \end{eqnarray*}
 2. $u(x,t)$ converges as $t \rightarrow \infty$ to a space-like surface $u_{\infty}$ (unique up to translation) which moves at a constant speed $c_{3}$ given by (\ref{2.15});\\
 3. if $\int_{\partial\Omega} \phi = 0$ then $c_{3}=0$, hence $u_{\infty}$ is a maximal space-like surface in the Lorentz manifold $M^{2}\times\mathbb{R}$.
\end{theorem}

\begin{remark}
\rm{ Clearly, if $M^{2}\equiv\mathbb{R}^{2}$, Theorem \ref{main1.1}
would give the existence of translating solutions to the space-like
nonparametric MCF with the prescribed contact angle boundary
condition in the Minkowski $3$-space $\mathbb{R}^{2,1}$. Besides, it
is worth pointing out one thing here, that is, it might be a little
surprise to readers that in order to get the existence of
translating solutions, our assumption here (only the strictly convex
assumption for the bounded domain $\Omega$) is \emph{weaker} than
that in \cite[Theorem 1.2]{aw}.
 }
\end{remark}

The paper is organized as follows. The uniform estimates for the
time derivative and the gradient of the solution to the IVP
$(\sharp)$ have been given in Section 2, which can be used to get
the solvability of the BVP $(\ast)$, the elliptic version of
$(\sharp)$, and the long-time existence of the IVP $(\sharp)$. The
existence of translating solutions to $(\sharp)$ has been shown in
Section 3.

\section{Estimates}

\subsection{The boundary}

\

Let $\{\theta,r\}$, with $r(x)$ the Riemannian distance function
$d(x,\partial\Omega)$ from $x$ to the boundary $\partial\Omega$, be
the local coordinates for a thin neighborhood of $\partial\Omega$
such that
\begin{eqnarray*}
\left\langle\frac{\partial}{\partial{r}},\frac{\partial}{\partial{\theta}}\right\rangle=0,
\quad
\left\langle\frac{\partial}{\partial{r}},\frac{\partial}{\partial{r}}\right\rangle=1,
\end{eqnarray*}
and
\begin{eqnarray*}
\left\{\frac{\partial}{\partial{r}},\frac{\partial}{\partial{\theta}}\right\}_{\partial\Omega}=\{N,T\}.
\end{eqnarray*}
Define a function $\varphi$ such that
$\left|\varphi^{-1}\frac{\partial}{\partial{\theta}}\right|^{2}=1$.
Then one can get the extended normal and tangent vectors to be the
orthonormal frame
$\left\{\frac{\partial}{\partial{r}},\varphi^{-1}\frac{\partial}{\partial{\theta}}\right\}$
of the thin neighborhood of $\Omega$, which, with the abuse of
notations, is also denoted by $\{N,T\}$. That is,
$\left\{\frac{\partial}{\partial{r}},\varphi^{-1}\frac{\partial}{\partial{\theta}}\right\}=\{N,T\}$.
By \cite[Lemma 2.1]{aw}, we have

\begin{lemma} \label{lemma2.1}
On $\partial\Omega$, one has
\\(i) $\nabla_{T}T=kN,\quad \nabla_{T}N=-kT,\quad \nabla_{N}T=\nabla_{N}N=0 $;
\\(ii) for any $f\in C^{\infty}(\bar{\Omega})$, $D_{N}D_{T}f=D_{T}D_{N}f+kD_{T}f$.\\
\end{lemma}

From the boundary condition of ($\sharp$), it is not hard to verify
the following facts
\begin{equation} \label{2.1}
|D_{N}u|^{2}=\phi^{2}v^{2},
\end{equation}
\begin{equation} \label{2.2}
|D_{T}u|^{2}=1-(1+\phi^{2})v^{2}.
\end{equation}
Differentiating conditions $(\ref{2.1})-(\ref{2.2})$ in the time and
tangential direction, we know that all the derivatives of $u$ on
$\partial\Omega$, except $D_{N}D_{N}U$, can be given in terms of the
first derivatives of $u$. More precisely, we have
\begin{equation} \label{2.3}
D_{N}u_{t}=-\phi\frac{D{u}Du_{t}}{\sqrt{1-|D{u}|^{2}}},
\end{equation}
\begin{equation} \label{2.4}
D_{T}D_{N}u=\phi D_{T}v+v D_{T}\phi,
\end{equation}\begin{equation} \label{2.5}
D_{N}D_{T}u=\phi D_{T}v+v D_{T}\phi+kD_{T}u,
\end{equation}
\begin{equation} \label{2.6}
D_{T}D_{T}u=-\frac{v(1+\phi^{2})D_{T}v +v^{2}\phi
D_{T}\phi}{D_{T}u}.
\end{equation}
Besides, elements of the inverse of the metric matrix of the
space-like graphs in $M^{2}\times\mathbb{R}$ are given by
\begin{equation} \label{2.7}
g^{TT}=\frac{1-(D_{N}u)^{2}}{v^{2}},
\end{equation}
\begin{equation} \label{2.8}
g^{NN}=1+\phi^{2},
\end{equation}
\begin{equation} \label{2.9}
g^{NT}=g^{TN}=\frac{D_{N}u D_{T}u}{v^{2}}.
\end{equation}

\subsection{Existence of solutions}

\

The key point of the existence for small time and the uniqueness of
solutions to IVP ($\sharp$) is to show that the evolution equation
in ($\sharp$) is uniformly parabolic at $t=0$, which can be assured
by the assumption that the initial graphic surface over $\Omega$ is
space-like. In fact, by the linearization theory (see \cite{ltu})
and the inverse function theorem (see, e.g., \cite{s}), together
with the space-like graphic assumption, the short-time existence and
the uniqueness of solutions to IVP ($\sharp$) can be obtained.

Assume that IVP ($\sharp$) has smooth solutions on the time interval
$[0,\mathbb{T}]$, which means all derivatives of $u$ have bounds on
$[0,\mathbb{T}]$. In the following, we establish a \emph{time
independent} priori estimate for the gradient of the solution (see
Theorem \ref{main2.3}), which leads to the preserving space-like
property for the evolving graphic surfaces in
$M^{2}\times\mathbb{R}$, and then turn the quasilinear evolution
equation into a uniformly parabolic equation. Furthermore, by the
standard theory of the second-order parabolic PDE, the higher order
regularity follows, which leads to the long-time existence of smooth
solutions of the IVP ($\sharp$).

\subsection{The time derivative estimate}

\

By the maximum principe of the second-order parabolic PDE, we have
the following.

\begin{lemma} \label{lemma2.2}
$\sup\limits_{\overline{\Omega}\times[0,\mathbb{T}]}|u_{t}|^{2}=\sup\limits_{\Omega_{0}}|u_{t}|^{2}$.
That is, there exists some positive constant
$c_{0}=c_{0}(u_{0})\in\mathbb{R}^{+}$ such that for any
$(x,t)\in\overline{\Omega}\times[0,\mathbb{T}]$, we have
\begin{eqnarray*}
|u_{t}|^{2}(x,t)\leq c_{0}.
\end{eqnarray*}
\end{lemma}

\begin{proof}
We first show that the maximum of $u_{t}$ must occur on
$\left(\partial\Omega\times[0,\mathbb{T}]\right)\cup \Omega_{0}$.
Let $(g^{ij})^{'}$ be the differential of
$g^{ij}=g^{ij}(x,u,Du)=g^{ij}(x,z,p)$ with respect to $p$. By a
direct computation, we have
\begin{eqnarray*}
\frac{\partial}{\partial t} |u_{t}|^{2}&=&2u_{t}\frac{\partial u_{t}}{\partial t}\\
&=&
2u_{t}\left(\frac{\partial g^{ij}}{\partial t} u_{ij}+g^{ij}\frac{\partial u_{ij}}{\partial t}\right)\\
&=&
2u_{t}\left(\frac{\partial g^{ij}}{\partial u^{k}}\frac{\partial u^{k}}{\partial t}u_{ij}+g^{ij}D_{i}D_{j}u_{t}\right)\\
&=&
2u_{t} \frac{\partial g^{ij}}{\partial u^{k}}\frac{\partial u^{k}}{\partial t}u_{ij}+g^{ij}(D_{i}D_{j}|u_{t}|^{2}-2D_{i}u_{t}D_{j}u_{t})\\
&=&
2u_{t} \frac{\partial g^{ij}}{\partial u^{k}}\frac{\partial u^{k}}{\partial t}u_{ij}+g^{ij}D_{i}D_{j}|u_{t}|^{2}-2\langle Du_{t},Du_{t}\rangle_{g}\\
&=&
2u_{t}\frac{\partial g^{ij}}{\partial u_{k}}\frac{\partial(\sigma^{ki}u_{i})}{\partial t}u_{ij}+g^{ij}D_{i}D_{j}|u_{t}|^{2}-2\langle Du_{t},Du_{t}\rangle_{g}\\
&=&
\frac{\partial g^{ij}}{\partial u^{k}}\sigma^{ki}u_{ij}D_{i}|u_{t}|^{2}+g^{ij}D_{i}D_{j}|u_{t}|^{2}-2\langle Du_{t},Du_{t}\rangle_{g}\\
&=&D_{i}D_{j}u\left\langle(g^{ij})^{'},\nabla|u_{t}|^{2}\right\rangle_{\sigma}+g^{ij}D_{i}D_{j}|u_{t}|^{2}-2\langle
Du_{t},Du_{t}\rangle_{g},
\end{eqnarray*}
where
$D_{i}D_{j}u\left\langle(g^{ij})^{'},\nabla|u_{t}|^{2}\right\rangle_{\sigma}=\sum\limits_{i,j,k=1}^{2}
D_{i}D_{j}u\left(\frac{\partial g^{ij}}{\partial p^{k}}\nabla
_{k}|u_{t}|^{2}\right)$. The boundedness of all the coefficients of
the above evolution equation in the bounded domain
$\overline{\Omega}\times[0,\mathbb{T}]$ follows by the continuity,
which implies that
\begin{eqnarray*}
 \sup\limits_{\overline{\Omega}\times[0,\mathbb{T}]}|u_{t}|^{2}=\sup\limits_{\left(\partial\Omega\times[0,\mathbb{T}]\right)\cup
\Omega_{0}}|u_{t}|^{2}
\end{eqnarray*}
 by directly applying the weak maximum
principle.

Next, we expel the possibility that the maximum occurs at
$(\xi,\tau)\in\partial\Omega\times[0,\mathbb{T}]$. Assume that
$\max\limits_{\Omega\times\{t\}}|u_{t}|^{2}=|u_{t}|^{2}(\xi,\tau)>0$,
we have $\left(D_{T}u_{t}\right)(\xi,\tau)=0$. By (\ref{2.3}), it
follows that
\begin{eqnarray*}
(D_{N}u_{t})(\xi,\tau)&=&-\phi \frac{D_{N}uD_{N}u_{t}+D_{T}uD_{T}u_{t}}{\sqrt{1-|Du|^{2}}}(\xi,\tau)\\
&=&
-\phi \frac{D_{N}uD_{N}u_{t}}{\sqrt{1-|Du|^{2}}}(\xi,\tau)\\
&=& -\phi^{2}(D_{N}u_{t})(\xi,\tau),
\end{eqnarray*}
which implies $(1+\phi^{2})(D_{N}u_{t})(\xi,\tau)=0$, i.e.,
$(D_{N}u_{t})(\xi,\tau)=0$. Therefore, by the Hopf Lemma,
$\frac{\partial}{\partial t}|u_{t}|^{2}(\xi,\tau)\leq 0$, and then
the conclusion of Lemma \ref{lemma2.2} follows.
\end{proof}

\subsection{The gradient estimate}
\

First, we need the evolution equation of $|Du|^{2}$.
\begin{lemma}\label{Evo-Du}
We have  the evolution equation of $|Du|^{2}$ as follows
\begin{eqnarray*}
\frac{\partial}{\partial t}|Du|^{2}=
\frac{D_{k}|Du|^{2}D_{i}|Du|^{2}}{v^{2}}g^{ik}+g^{ij}D_{i}D_{j}|Du|^{2}-|D^{2}u|^{2}-|D|Du|^{2}|^{2}-K|Du|^2,
\end{eqnarray*}
where $K$ denotes the Gaussian curvature of $M^{2}$.
\end{lemma}

\begin{proof}
First, by
direct computations, we have
 \begin{eqnarray*}
\frac{\partial}{\partial t}|Du|^{2}=\frac{\partial}{\partial
t}(u^{m}u_{m})=2u^{m}(u_{m})_{t}=2u^{m}(u_{t})_{m},
\end{eqnarray*}
\begin{eqnarray*}
(u_{t})_{m}&=&(g^{ij}u_{ij})_{m}\\
&=& (g^{ij})_{m}u_{ij}+g^{ij}(u_{ij})_{m},
\end{eqnarray*}
\begin{eqnarray*}
(g^{ij})_{m}&=&\left(\sigma^{ij}+\frac{u^{i}u^{j}}{1-|Du|^{2}}\right)_{m}\\
&=&
\frac{(u^{i})_{m}u^{j}}{1-|Du|^{2}}+\frac{u^{i}(u^{j})_{m}}{1-|Du|^{2}}+\frac{2u^{k}(u_{k})_{m}u^{i}u^{j}}{(1-|Du|^{2})^{2}},
\end{eqnarray*}
and
 \begin{eqnarray*}
(|Du|^{2})_{ij}=2\sigma^{mk}u_{kj}u_{mi}+2\sigma^{mk}u_{k}u_{mij}.
\end{eqnarray*}
Therefore, he evolution equation for the  gradient is given as
follows
\begin{eqnarray*}
\frac{\partial}{\partial t}|Du|^{2}&=&2u^{m}\left[\frac{(u^{i})_{m}u^{j}}{1-|Du|^{2}}+\frac{u^{i}(u^{j})_{m}}{1-|Du|^{2}}+\frac{2u^{k}(u_{k})_{m}u^{i}u^{j}}{(1-|Du|^{2})^{2}}\right]u_{ij}+2g^{ij}u^{m}(u_{ij})_{m}\\
&=&
2u^{m}\left[\frac{(\sigma^{ik}u_{k})_{m}u^{j}}{v^{2}}+\frac{u^{i}(\sigma^{jk}u_{k})_{m}}{v^{2}}+\frac{2u^{k}(u_{k})_{m}u^{i}u^{j}}{v^{4}}\right]u_{ij}+2g^{ij}u^{m}(u_{ij})_{m}\\
&=&
2u^{m}\left(\frac{\sigma^{ik}u_{km}u^{j}}{v^{2}}+\frac{u^{i}\sigma^{jk}u_{km}}{v^{2}}+\frac{2u^{k}u_{km}u^{i}u^{j}}{v^{4}}\right)u_{ij}+2g^{ij}u^{m}(u_{ij})_{m}\\
&=&
2u^{m}\frac{u_{km}}{v^{2}}\left(\sigma^{ik}u^{j}+\sigma^{jk}u^{i}+\frac{2u^{k}u^{i}u^{j}}{v^{2}}\right)u_{ij}\\
&& \qquad +g^{ij}\left(D_{i}D_{j}|Du|^{2}-2\sigma^{mk}u_{kj}u_{mi}-2u^{m}u_{l}R^{l}_{imj}\right)\\
&=&
\frac{2u^{m}u_{km}\sigma^{ik}u^{j}u_{ij}}{v^{2}}+\frac{2u^{m}u_{km}\sigma^{ik}u^{j}u_{ij}}{v^{2}}+\frac{2u^{m}2^{k}u_{km}u^{i}u^{j}u_{ij}}{v^{4}}\\
&&+g^{ij}D_{i}D_{j}|Du|^{2}-2g^{ij}\sigma^{mk}u_{kj}u_{mi}-2g^{ij}u^{m}u^{l}R_{limj}\\
&=&
\frac{4u^{m}u_{km}\sigma^{ik}u^{j}u_{ij}}{v^{2}}+\frac{4u^{m}u^{k}u_{km}u^{i}u^{j}u_{ij}}{v^{4}}\\
&&+
g^{ij}D_{i}D_{j}|Du|^{2}-2g^{ij}\sigma^{mk}u_{kj}u_{mi}-2g^{ij}u^{m}u^{l}R_{limj}\\
&=&
\frac{4u^{m}u_{km}u^{j}u_{ij}}{v^{2}}\left(\sigma^{ik}+\frac{u^{k}u^{i}}{v^{2}}\right)+g^{ij}D_{i}D_{j}|Du|^{2}-2g^{ij}\sigma^{mk}u_{kj}u_{mi}-2g^{ij}u^{m}u^{l}R_{limj}\\
&=&
\frac{2u^{m}u_{km}2u^{j}u_{ij}}{v^{2}}g^{ik}+g^{ij}D_{i}D_{j}|Du|^{2}-2g^{ij}\sigma^{mk}u_{kj}u_{mi}-2g^{ij}u^{m}u^{l}R_{limj}\\
&=&
\frac{D_{k}|Du|^{2}D_{i}|Du|^{2}}{v^{2}}g^{ik}+g^{ij}D_{i}D_{j}|Du|^{2}-2g^{ij}\sigma^{mk}u_{kj}u_{mi}-2g^{ij}u^{m}u^{l}R_{limj}\\
&=&
\frac{D_{k}|Du|^{2}D_{i}|Du|^{2}}{v^{2}}g^{ik}+g^{ij}D_{i}D_{j}|Du|^{2}-2\left(\sigma^{ij}+\frac{u^{i}u^{j}}{v^{2}}\right)\sigma^{mk}u_{kj}u_{mi}\\
&&-2\left(\sigma^{ij}+\frac{u^{i}u^{j}}{v^{2}}\right)u^{m}u^{l}R_{limj}\\
&=&
\frac{D_{k}|Du|^{2}D_{i}|Du|^{2}}{v^{2}}g^{ik}+g^{ij}D_{i}D_{j}|Du|^{2}-|D^{2}u|^{2}-2\frac{u^{i}u_{im}u^{j}u_{jk}\sigma^{mk}}{v^{2}}\\
&&-2\sigma^{ij}u^{m}u^{l}R_{limj}\\
&=&
\frac{D_{k}|Du|^{2}D_{i}|Du|^{2}}{v^{2}}g^{ik}+g^{ij}D_{i}D_{j}|Du|^{2}-|D^{2}u|^{2}-|D|Du|^{2}|^{2}-2\sigma
^{ij}u^{m}u^{l}R_{limj}\\
&=&
\frac{D_{k}|Du|^{2}D_{i}|Du|^{2}}{v^{2}}g^{ik}+g^{ij}D_{i}D_{j}|Du|^{2}-|D^{2}u|^{2}-|D|Du|^{2}|^{2}-K|Du|^2,
\end{eqnarray*}
where $R_{limj}$, $1\leq l,i,m,j\leq2$, are the components of the
curvature tensor  on $M^{2}$.
\end{proof}

Then, we begin to estimate the gradient of $u$ by a clever use of the boundary
algebra.

\begin{theorem}\label{main2.3}
Under the assumptions Theorem \ref{main1.1}, there exists a positive
constant $c_{1}=c_{1}(u_{0},\kappa_{0},\phi_{0},\phi_{1},\phi_{2})$
such that
\begin{eqnarray*}
\sup\limits_{\overline{\Omega}\times[0,\mathbb{T}]}|Du|^{2}\leq
c_{1}<1.
\end{eqnarray*}
\end{theorem}

\begin{proof}
We first show that the maximum of $|Du|^{2}$ must occur on
$\left(\partial\Omega\times[0,\mathbb{T}]\right)\cup \Omega_{0}$.
Applying Lemma \ref{Evo-Du}, we can get the following estimate
\begin{eqnarray*}
\frac{\partial}{\partial
t}|Du|^{2}\leq\frac{D_{k}|Du|^{2}D_{i}|Du|^{2}}{v^{2}}g^{ik}+g^{ij}D_{i}D_{j}|Du|^{2},
\end{eqnarray*}
since $M^{2}$ has nonnegative Gaussian curvature.
Applying the weak maximum principle to the above evolution
inequality, we have
 \begin{eqnarray*}
\sup\limits_{\overline{\Omega}\times[0,\mathbb{T}]}|Du|^{2}=\sup\limits_{\left(\partial\Omega\times[0,\mathbb{T}]\right)\cup
\Omega_{0}}|Du|^{2}.
 \end{eqnarray*}
If the maximum of $|Du|^{2}$ occurs at $\Omega_{0}$, then
$\sup\limits_{\overline{\Omega}\times[0,\mathbb{T}]}|Du|^{2}\leq
\sup\limits_{\Omega_{0}}|Du|^{2}<1$, where the last inequality holds
since $\{(x,u(x,0))|x\in M^{2}\}$ is a space-like graph of
$M^{2}\times\mathbb{R}$. Now, assume that the maximum of $|Du|^{2}$
occurs at $(\xi,\tau)\in\partial\Omega\times[0,\mathbb{T}]$. We
divide the argument into two cases:

Case (1). If $|D_{T}u|(\xi,\tau)\leq\frac{1}{2}$,
 then applying the fact
$(1+\phi^{2})v^{2}=1-|D_{T}u|^{2}$, we have
\begin{eqnarray*}
|Du|^{2}(\xi,\tau)\leq1-\frac{3}{4(1+\phi^{2})}<1,
\end{eqnarray*}
which establishes an upper bound for $|Du|^{2}$ on
$\overline{\Omega}\times[0,\mathbb{T}]$ already.

Case (2). If $|D_{T}u|(\xi,\tau)>\frac{1}{2}$, then at $(\xi,\tau)$,
one has
$$D_{N}|Du|^{2}(\xi,\tau)\leq0,$$
$$D_{T}|Du|^{2}(\xi,\tau)=0=D_{T}v(\xi,\tau).$$
Therefore, at $(\xi,\tau)$, (\ref{2.4})-(\ref{2.6}) can be
simplified as follows
\begin{equation} \label{2.10}
D_{T}D_{N}u=v D_{T}\phi,
\end{equation}
\begin{equation} \label{2.11}
D_{N}D_{T}u=v D_{T}\phi+\kappa D_{T}u,
\end{equation}
\begin{equation} \label{2.12}
D_{T}D_{T}(u)=\frac{-v^{2}\phi D_{T}\phi}{D_{T}u}.
\end{equation}
Our target is to show the following
\begin{equation} \label{2.13}
(D_{N}u)(D_{N}D_{N}u(\xi,\tau))+(D_{T}u)(D_{N}D_{T}U)(\xi,\tau)\leq0.
\end{equation}
In order to get this, we need to consider the evolution equation of
$u$. In fact, by using the assumption that $|Du|^{2}$ gets its
maximum at $(\xi,\tau)$, (\ref{2.1})-(\ref{2.3}), and
(\ref{2.10})-(\ref{2.12}), we get that at $(\xi,\tau)$, the
following identity
\begin{eqnarray*}
u_{t}&=&g^{TT}D_{T}D_{T}u+g^{TN}D_{T}D_{N}u+g^{NT}D_{N}D_{T}u+g^{NN}D_{N}D_{N}u\\
&&-
g^{TT}\langle D_{T}T,Du\rangle-g^{TN}\langle D_{T}N,Du\rangle-g^{NT}\langle D_{N}T,Du\rangle-g^{NN}\langle D_{N}N,Du\rangle\\
&=&
\frac{1- |D_{N}u|^{2}}{v^{2}}\cdot\frac{\left(-v D_{T}v(1+\phi^{2})-v^{2}\phi D_{T}\phi\right)}{D_{T}u}+\frac{D_{N}uD_{T}u}{v^{2}}(D_{T}\phi\cdot v+\phi\cdot D_{T}v)\\
&&+\frac{D_{N}uD_{T}u}{v^{2}}(D_{T}\phi\cdot v+\phi\cdot
D_{T}v+\kappa D_{T}u)+(1+\phi^{2})D_{N}D_{N}u-\frac{1-|D_{N}u|^{2}}{v^{2}}\langle \kappa N,Du\rangle\\
&&-\frac{D_{N}uD_{T}u}{v^{2}}\langle -\kappa T,Du\rangle\\
&=& \frac{1-\phi^{2}v^{2}}{v^{2}}\left[\frac{-v\cdot
D_{T}v(1+\phi^{2})-v^{2}\phi
D_{T}\phi}{D_{T}u}-\kappa\phi v \right]+(1+\phi^{2})D_{N}D_{N}u
\\&&+ 2\frac{\phi v D_{T}u}{v^{2}}(D_{T}\phi\cdot v+\phi D_{T}v+\kappa\cdot D_{T}u)\\
&=&
\frac{1-\phi^{2}v^{2}}{v^{2}}\left(\frac{-v^{2}\phi
D_{T}\phi}{D_{T}u}-\kappa\phi v\right)+(1+\phi^{2})D_{N}D_{N}u
+2\frac{\phi v
D_{T}u}{v^{2}}(D_{T}\phi\cdot v+\kappa\cdot D_{T}u)
\end{eqnarray*}
holds. That is,
\begin{eqnarray*}
(1+\phi^{2})D_{N}D_{N}u&=&u_{t}-\frac{1-\phi^{2}v^{2}}{v^{2}}\left(\frac{-v^{2}\phi
D_{T}\phi}{D_{T}u}-\kappa\phi v\right)-2\frac{\phi v
D_{T}u}{v^{2}}(D_{T}\phi \cdot v+\kappa\cdot D_{T}u)\\
&=& u_{t}+\frac{\phi
D_{T}\phi}{D_{T}u}\left(1-\phi^{2}
v^{2}-2|D_{T}u|^{2}\right)+\frac{(1-\phi^{2}v^{2})\kappa\phi}{v}-\frac{2\kappa\phi|D_{T}u|^{2}}{v}\\
&=& u_{t}+\frac{\phi
D_{T}\phi}{D_{T}u}\left(v^{2}-|D_{T}u|^{2}\right)-\frac{2\kappa\phi
|D_{T}u|^{2}}{v}+\frac{\kappa\phi(1-\phi^{2}v^{2})}{v}.
\end{eqnarray*}
Substituting the above identity into (\ref{2.13}), together with
(\ref{2.11}), yields
\begin{eqnarray*}
\phi v \left[u_{t}+\frac{\phi
D_{T}\phi}{D_{T}}(v^{2}-|D_{T}u|^{2})-\frac{2\kappa\phi
u_{T}^{2}}{v}+\frac{\kappa\phi(1-\phi^{2}v^{2})}{v}\right]+(1+\phi^{2})D_{T}u(v
D_{T}\phi+\kappa u_{T})\leq0,
\end{eqnarray*}
which is equivalent with
\begin{eqnarray} \label{2.14}
\phi v u_{t}+\frac{\phi^{2}v
D_{T}\phi}{D_{T}u}\left(v^{2}-|D_{T}u|^{2}\right)-2\kappa\phi^{2}|D_{T}u|^{2}+\kappa\phi^{2}(1-\phi^{2}v^{2})+\nonumber\\
(1+\phi^{2})D_{T}u(v D_{T}\phi+\kappa D_{T}u)\leq 0.
\qquad\qquad
\end{eqnarray}
It is easy to verify that
\begin{eqnarray*}
-2\kappa\phi^{2}|D_{T}u|^{2}+\kappa\phi^{2}(1-\phi^{2}v^{2})+(1+\phi^{2})\kappa
|D_{T}u|^{2}=\kappa(1-v^{2}),
\end{eqnarray*}
and
\begin{eqnarray*}
\frac{\phi^{2}v
D_{T}\phi}{D_{T}u}(v^{2}-|D_{T}u|^{2})+(1+\phi^{2})D_{T}uD_{T}\phi v=\frac{v
D_{T}\phi}{D_{T}u}(1-v^{2}).
\end{eqnarray*}
Using the above two identities, (\ref{2.14}) can be simplified as
follows
\begin{eqnarray*}
\phi v u_{t}+\kappa(1-v^{2})+\frac{v
D_{T}\phi}{D_{T}u}(1-v^{2})\leq 0.
\end{eqnarray*}
Note that
\begin{eqnarray*}
1-v^2=\frac{\phi^2}{1+\phi^2}+\frac{|D_{T}u|^2}{1+\phi^2}.
\end{eqnarray*}
Hence,
\begin{eqnarray*}
\phi v u_{t}+\kappa(1-v^{2})+\frac{v
D_{T}\phi}{D_{T}u}\frac{\phi^2}{1+\phi^2}+\frac{vD_T\phi D_T u}{1+\phi^2}\leq 0.
\end{eqnarray*}
In view of Lemma \ref{lemma2.2}, $|D_{T}u|(\xi,\tau)>\frac{1}{2}$
and $\Omega$ is strictly convex, we can find $\kappa_0$ such that
\begin{eqnarray*}
0<\kappa_0(1-v^2)\leq \kappa(1-v^2)\leq c_{2}v,
\end{eqnarray*}
where $c_{2}$ is a positive constant depending on $c_{0}$,
$\phi_{0}$, $\phi_{1}$, $\phi_{2}$. Therefore£¬
$$|Du|^{2}\leq c_{1}:=\frac{\sqrt{c_{2}^{4}+4c_{2}^{2}\kappa_{0}^{2}}-c_{2}^{2}}{2\kappa_{0}^{2}}<1.$$
Our proof is finished.
\end{proof}

\subsection{Boundary value problems}
\

Applying the above gradient estimate, Theorem \ref{main2.3}, one can
solve the following boundary value problem (BVP for short)
\begin{eqnarray*}
 (\ast)\qquad\left\{
\begin{array}{lll}
\left(\sigma^{ij}+\frac{D^{i}uD^{j}u}{1-|Du|^{2}}\right)D_{i}D_{j}u=c_{3}  \qquad \qquad &{\rm{on}} \quad \Omega\\
D_{N}u=\phi(x)v \qquad \qquad &{\rm{on}} \quad \partial\Omega,
\end{array}
\right.
\end{eqnarray*}
where $c_{3}$ is a constant determined uniquely by (\ref{2.15})
below. Clearly, the BVP $(\ast)$ can be seen as the elliptic version
of IVP $(\sharp)$.

In fact, since the LHS of the first equation in BVP $(\ast)$ can be
written as
$$\sqrt{1-|Du|^{2}}D_{i}\left(\frac{D_{i}u}{\sqrt{1-|Du|^{2}}}\right),$$
 integrating by parts one can easily get
\begin{equation} \label{2.15}
c_{3}=-\frac{\int_{\partial \Omega}\phi
}{\int_{\Omega}(1-|Du|^{2})^{-\frac{1}{2}}},
\end{equation}
where, for convenience, we have dropped volume elements of the
domain $\Omega$ and its boundary $\partial\Omega$ simultaneously.

One method for solving BVP $(\ast)$ is to consider the solvability
the following BVP.

\begin{eqnarray*}
(\ast\ast)\qquad \left\{
\begin{array}{lll}
\left(\sigma^{ij}+\frac{D^{i}u_{\varepsilon}D^{j}u_{\varepsilon}}{1-|Du_{\varepsilon}|^{2}}\right)D_{i}D_{j}u_{\varepsilon}=\varepsilon
u_{\varepsilon} \qquad \qquad &{\rm{on}}\quad \Omega\\
D_{N}u_{\varepsilon}=\phi(x)\sqrt{1-|Du_{\varepsilon}|^{2}} \qquad
\qquad &{\rm{on}}\quad \partial\Omega.
\end{array}
\right.
\end{eqnarray*}

\begin{theorem}
The BVP $(\ast)$ has a unique, smooth solution.
\end{theorem}

\begin{proof}
We will use an argument similar to those in \cite{aw,ltu}. For BVP
$(\ast\ast)$, it is known that it has solutions for $\varepsilon>0$.
Therefore, one can replace $u_{t}$ with $\varepsilon
u_{\varepsilon}$ in the gradient estimate of Theorem \ref{main2.3}
and get a conclusion that a limit solution to $(\ast\ast)$ exists as
$\varepsilon\rightarrow0$, provided there exists some $c_{0}$,
independent of $\varepsilon$, such that $|\varepsilon
u_{\varepsilon}|^{2}\leq c_{0}$.

Let $\psi$ be a smooth function defined on $\Omega$ satisfying
$D_{N}\psi<\phi\sqrt{1-|D\psi|^{2}}$ on $\partial\Omega$. This kind
of smooth functions can always be constructed. For instance, let $d$
be the distance function to $\partial\Omega$ and $A$ be a constant
such that $\frac{A}{\sqrt{1-A^{2}}}<\phi$ on $\partial\Omega$. It is
easy to check that a function $\psi$ defined to be $Ad$ near
$\partial\Omega$ and extended to be a smooth function on all of
$\Omega$ would satisfy the requirements that $\psi\in
C^{\infty}(\partial\Omega)$, $D_{N}\psi<\phi\sqrt{1-|D\psi|^{2}}$.
Assume that $\psi-u_{\varepsilon}$ attains its minimum at some point
$\xi\in\Omega$.

If $\xi\in\partial\Omega$, then
$D_{T}\psi(\xi)=D_{T}u_{\varepsilon}(\xi)$ and $D_{N}\psi(\xi)\geq
D_{N}u_{\varepsilon}(\xi)$. One can get
\begin{eqnarray*}
\phi(\xi)>\frac{D_{N}\psi}{\sqrt{1-|D_{T}\psi|^{2}-|D_{N}\psi|^{2}}}(\xi)\geq\frac{D_{N}u_{\varepsilon}}{\sqrt{1-|D_{T}u_{\varepsilon}|^{2}-|D_{N}u_{\varepsilon}|^{2}}}(\xi)=\phi(\xi),
\end{eqnarray*}
since the function $\frac{q}{\sqrt{1-b^{2}-q^{2}}}$ with $b$ a fixed
constant is monotone nondecreasing in $q$. This is a contradiction.

Therefore, $\xi\in\Omega$, $D\psi(\xi)=Du_{\varepsilon}(\xi)$ and
$D^{2}\psi(\xi)\geq D^{2}u_{\varepsilon}(\xi)$. There exists a
constant $c_{4}=c_{4}(\psi)$ such that
\begin{eqnarray*}
c_{4}\geq\left(\sigma^{ij}+\frac{D^{i}\psi(\xi)
D^{j}\psi(\xi)}{1-|D\psi(\xi)|^{2}}\right)D_{i}D_{j}\psi(\xi)\geq
\left(\sigma^{ij}+\frac{D^{i}u_{\varepsilon}(\xi)D^{j}u_{\varepsilon}(\xi)}{1-|Du_{\varepsilon}(\xi)|^{2}}\right)D_{i}D_{j}u_{\varepsilon}(\xi)=\varepsilon
u_{\varepsilon}(\xi).
\end{eqnarray*}
Together with the fact that $\varepsilon\psi(z)-\varepsilon
u_{\varepsilon}(z)\geq\varepsilon\psi(\xi)-\varepsilon
u_{\varepsilon}(\xi)$ for any $z\in\Omega$, we have
\begin{eqnarray*}
\varepsilon
u_{\varepsilon}(z)\leq\varepsilon\psi(z)-\varepsilon\psi(\xi)+\varepsilon
u_{\varepsilon}(\xi)\leq\varepsilon\psi(z)-\varepsilon\psi(\xi)+c_{4}
\end{eqnarray*}
for any $z\in\Omega$. By a similar barrier argument, one can get a
lower bound for $\varepsilon u_{\varepsilon}$. As in \cite{ltu},
$|Du_{\varepsilon}|^{2}\leq c_{1}$ implies $|D(\varepsilon
u_{\varepsilon})|^{2}\rightarrow 0$ as $\varepsilon\rightarrow 0$,
and then we have $\varepsilon u_{\varepsilon}\rightarrow c_{3}$.
This gives the existence of solutions to BVP $(\ast)$.

Now, in what follows, we would like to show the uniqueness of the
solutions. Assume that the BVP $(\ast)$ has two solutions $u_{1}$,
$u_{2}$ with constants $c_{5}$, $c_{6}$ on the RHS of $(\ast)$ and
$c_{5}<c_{6}$. Without loss of generality, assume $u_{1}\geq u_{2}$.
By the linearization process, one easily knows that
$\mathcal{U}:=u_{1}-u_{2}$ satisfies a linear elliptic differential
inequality $\mathcal{L}(\mathcal{U})<0$. By the maximum principle,
the minimum of $\mathcal{U}$ must be achieved at some point
$\zeta\in\partial\Omega$, which implies that
$|D_{T}u_{1}|^{2}(\zeta)=|D_{T}u_{2}|^{2}(\zeta)=a^{2}$ for some
$a\in\mathbb{R}^{+}$. Since
\begin{eqnarray*}
\frac{D_{N}u_{1}}{\sqrt{1-a^{2}-|D_{N}u_{1}|^{2}}}(\zeta)=\frac{D_{N}u_{2}}{\sqrt{1-a^{2}-|D_{N}u_{2}|^{2}}}(\zeta),
\end{eqnarray*}
it follows that $D_{N}u_{1}(\zeta)=D_{N}u_{2}(\zeta)$ at
$\zeta\in\partial\Omega$ by using the fact that the function
$\frac{q}{\sqrt{1-a^{2}-q^{2}}}$ is monotone nondecreasing in $q$.
However, this is contradict with the Hopf boundary point lemma. So,
$c_{5}\geq c_{6}$. Reversing the roles of $c_{5}$ and $c_{6}$, one
has $c_{5}\leq c_{6}$. Therefore, one can get $c_{5}=c_{6}$. By a
similar argument, one can also obtain $u_{1}=u_{2}$. This gives the
uniqueness of solutions to BVP $(\ast)$.

Our proof is finished.
\end{proof}

\begin{remark} \label{remark2.1}
\rm{Clearly, if $u=u(x)$ is a solution to the BVP $(\ast)$, then
$\widetilde{u}(x,t)=u(x)+c_{3}t$ is a solution to the IVP
$(\sharp)$. That is to say $\widetilde{u}$ is a translating solution
with constant speed $|c_{3}|$.}
\end{remark}

\section{Convergence}

Now, we can show the following uniqueness conclusion of limit
solutions to the IVP $(\sharp)$ (up to translation) by applying the
strong maximum principle of the second-order linear parabolic PDE.

\begin{lemma} \label{lemma3.1}
Let $\widetilde{u}_{1}$ and  $\widetilde{u}_{2}$ be any two solution
to IVP $(\sharp)$ and let
$\widetilde{\mathcal{U}}=\widetilde{u}_{1}-\widetilde{u}_{2}$. Then
$\widetilde{\mathcal{U}}$ becomes a constant function as
$t\rightarrow\infty$. In particular, if $u$ is a solution to the BVP
$(\ast)$, then all limit solutions to the IVP $(\sharp)$ are of the
form $u+c_{3}t$.
\end{lemma}

\begin{proof}
By the linearization process, one can easily get that
$\widetilde{\mathcal{U}}$ satisfies the following linear parabolic
equation
\begin{eqnarray*}
\frac{\partial}{\partial
t}\widetilde{\mathcal{U}}=\widetilde{g}^{ij}D_{i}D_{j}\widetilde{\mathcal{U}}+\widetilde{b}^{i}D_{i}\widetilde{\mathcal{U}},
\qquad \rm{on} \quad \Omega\times[0,T]
\end{eqnarray*}
with the boundary condition
\begin{eqnarray*}
0=\left\langle\frac{Du_{1}}{\sqrt{1-|Du_{1}|^{2}}}-\frac{Du_{2}}{\sqrt{1-|Du_{2}|^{2}}},N\right\rangle:=\widetilde{c}^{ij}N_{j}D_{i}\widetilde{\mathcal{U}},
\qquad \rm{on} \quad \partial\Omega\times[0,T],
\end{eqnarray*}
where
 \begin{eqnarray*}
  \widetilde{g}^{ij}=\int^{1}_{0}g^{ij}\left(\theta
  Du_{1}+(1-\theta)Du_{2}\right)d\theta,
\end{eqnarray*}
and $\widetilde{b}^{i}$, $\widetilde{c}^{ij}$ are similarly
determined (see, e.g., \cite{gt}), $N_{j}$'s are components of the
unit normal vector $N$. Note that $\widetilde{c}^{ij}$ is a positive
definite matrix. By the strong maximum principle, we know that the
oscillation function
${\rm{osc}}(t):=\max\widetilde{\mathcal{U}}(\cdot,t)-\min\widetilde{\mathcal{U}}(\cdot,t)\geq0$
is strictly decreasing in $t$ unless $\widetilde{\mathcal{U}}$ is
constant.

The long-time existence of solutions to the IVP $(\sharp)$ has been
explained in Subsection 2.3 provided the time-independent priori
gradient estimate can be obtained. Therefore, we have
$\mathbb{T}=\infty$ here.

We claim that $|\widetilde{\mathcal{U}}|$ must be uniformly bounded
on $\overline{\Omega}\times[0,\infty)$. By the maximum principle, we
know that the minimum of $\widetilde{\mathcal{U}}$ should be
achieved at some point
$(\xi,t_{0})\in\left(\partial\Omega\times[0,\infty)\right)\cup
\Omega_{0}$. If $(\xi,t_{0})\in\partial\Omega\times[0,\infty)$, then
$D_{T}\widetilde{\mathcal{U}}(\xi,t_{0})=0$ and
$D_{N}\widetilde{\mathcal{U}}(\xi,t_{0})\geq 0$. That is,
$D_{T}u_{1}(\xi,t_{0})=D_{T}u_{1}(\xi,t_{0})$ and
$D_{N}u_{1}(\xi,t_{0})\geq D_{N}u_{1}(\xi,t_{0})$. Therefore, one
has
\begin{eqnarray*}
\frac{D_{N}u_{1}}{\sqrt{1-|D_{T}u_{1}|^{2}-|D_{N}u_{1}|^{2}}}(\xi,t_{0})>\frac{D_{N}u_{2}}{\sqrt{1-|D_{T}u_{2}|^{2}-|D_{N}u_{2}|^{2}}}(\xi,t_{0}),
\end{eqnarray*}
since the function $\frac{q}{\sqrt{1-b^{2}-q^{2}}}$ with $b$ a fixed
constant is strictly increasing in $q$. However, this is contradict
with the boundary condition
\begin{eqnarray*}
\frac{D_{N}u_{1}}{\sqrt{1-|D_{T}u_{1}|^{2}-|D_{N}u_{1}|^{2}}}(\xi,t_{0})=\frac{D_{N}u_{2}}{\sqrt{1-|D_{T}u_{2}|^{2}-|D_{N}u_{2}|^{2}}}(\xi,t_{0})=\phi(\xi).
\end{eqnarray*}
Therefore, $(\xi,t_{0})\in\Omega_{0}$, i.e., $\xi\in\Omega$ and
$t_{0}=0$. This is to say that $\widetilde{\mathcal{U}}$ attains its
minimum on $\Omega_{0}$. The same situation happens to the maximum
of $\widetilde{\mathcal{U}}$. Hence, we have
$|\widetilde{\mathcal{U}}|=|u_{1}-u_{2}|\leq c_{7}(u_{0})$ for some
nonnegative constant $c_{7}(u_{0})$ only depending on $u_{0}$.

Since $|\widetilde{\mathcal{U}}|$ is uniformly bounded on
$\overline{\Omega}\times[0,\infty)$, we can take a sequence
$\{t_{n}\}$, $n\in\mathbb{Z}^{+}$ with $\mathbb{Z}^{+}$ the set of
all positive integers, such that the limit
$\lim_{t_{n}\rightarrow\infty}\widetilde{\mathcal{U}}(\cdot,t_{n})$
exists. If
$\lim_{t\rightarrow\infty}\widetilde{\mathcal{U}}(\cdot,t)$ were not
a constant function, then a limit of
$\widetilde{\mathcal{U}}_{n}(\cdot,t):=\widetilde{\mathcal{U}}(\cdot,t+t_{n})$
as $t_{n}\rightarrow\infty$ would yield a solution on
$\Omega\times[0,\infty)$ which would not be constant but on which
${\rm{osc}}(t)$ would be constant. But this is contradict with the
strict monotonicity of ${\rm{osc}}(t)$. Therefore,
$\lim_{t\rightarrow\infty}\widetilde{\mathcal{U}}(\cdot,t)$ should
be a constant function, which implies the first assertion.

By Remark \ref{remark2.1}, we know that $u+c_{3}t$ is a solution to
the IVP $(\sharp)$ provided $u$ is a solution to the BVP $(\ast)$.
Hence, for any solution $\omega$ of $(\sharp)$, by the first
assertion, one has $\omega-(u+c_{3}t)$ tends to a constant as
$t\rightarrow\infty$, which implies that $\omega$ tends to
$u+c_{3}t$ for a different $t$. This completes the proof of the
second assertion.
\end{proof}

By applying Lemma \ref{lemma3.1} directly, we have the following.

\begin{corollary}\label{coro}
For a solution $u=u(x,t)$ of the IVP $(\sharp)$, there exists some
positive constant $c_{8}\in\mathbb{R}^{+}$ such that
$|u(x,t)-c_{3}t|\leq c_{8}$.
\end{corollary}

Now, we show that if $\int_{\partial\Omega}\phi=0$, the limiting
surface $u_{\infty}:=\lim_{t\rightarrow\infty}u(\cdot,t)$, with
$u(\cdot,t)$
 the solution to IVP $(\sharp)$, should be maximal space-like.

 \begin{lemma} \label{lemma3.3}
 If $\int_{\partial\Omega}\phi=0$, then $c_{3}=0$ and
 $\lim\limits_{t\rightarrow\infty}u_{t}=0$. That is, solutions
 $u(\cdot,t)$ to the IVP $(\sharp)$ converge to a maximal space-like surface $u_{\infty}$ in the Lorentz manifold $M^{2}\times\mathbb{R}$.
 \end{lemma}

\begin{proof}
The first assertion follows directly from (\ref{2.15}). By a direct
calculation, we have
\begin{eqnarray*}
\frac{d}{dt}\int_{\Omega}v=-\int_{\Omega}\frac{D_{i}u_{t}D_{i}u}{v}=\int_{\Omega}\frac{u_{t}^{2}}{v}+\int_{\partial\Omega}u_{t}\phi,
\end{eqnarray*}
which implies
\begin{eqnarray*}
\frac{d}{dt}\left(\int_{\Omega}v-\int_{\partial\Omega}u\phi\right)=\int_{\Omega}\frac{u_{t}^{2}}{v}.
\end{eqnarray*}
Applying Theorem \ref{main2.3} and Corollary \ref{coro}, we know
that there exists a positive constant $c_{9}\in\mathbb{R}^{+}$
depending on $c_{1}$, $c_{8}$ such that
\begin{eqnarray*}
\int_{\Omega}\frac{u_{t}^{2}}{v}\leq c_{9},
\end{eqnarray*}
which implies the second assertion of Lemma \ref{lemma3.3}.
\end{proof}

\vspace{0.2 cm}

$\\$\textbf{Acknowledgements}. This research was supported in part
by the National Natural Science Foundation of China (Grant Nos.
11401131 and 11101132), the Fok Ying-Tung Education Foundation
(China), and Hubei Key Laboratory of Applied Mathematics (Hubei
University). The last part of this paper was carried out when the
corresponding author, Prof. J. Mao, visited the Shanghai Center for
Mathematical Sciences (SCMS), Fudan University in April 2018, and he
is grateful to Prof. Peng Wu for the hospitality during his visit to
SCMS.


\vspace {1 cm}

\begin{thebibliography}{50}
\setlength{\itemsep}{-0pt} \small

\bibitem{aw} S.-J. Altschuler and L.-F. Wu, \emph{Translating surfaces of the non-parametric mean curvature flow with prescribed contact
angle}, Calc. Var. Partial Differential Equations {\bf 2} (1994)
101--111.

\bibitem{av1} S.-B. Angenent and J.-J.-L. Velazquez, \emph{Asymptotic shape of cusp singularities in curve shortening}, Duke
Math. J. {\bf 77} (1995) 71--110.

\bibitem{av2} S.-B. Angenent and J.-J.-L. Velazquez, \emph{Degenerate neckpinches in mean curvature
flow}, J. Reine Angew. Math. {\bf 482 } (1997) 15--66.


\bibitem{cm} L. Chen and J. Mao, \emph{Non-parametric inverse curvature flows in the AdS-Schwarzschild
manifold}, J. Geom. Anal. {\bf 28} (2018) 921--949.

\bibitem{cmxx} L. Chen, J. Mao, N. Xiang and C. Xu, \emph{Inverse mean curvature flow inside a cone in warped
products}, submitted and available online at arXiv:1705.04865v3.

\bibitem{gt} D. Gilbarg and N. Trudinger, \emph{Elliptic Partial Differential
Equations}, Grundlehren Math. Wiss., vol. 224, Springer, Berlin,
Heidelberg, New York, 1983.

\bibitem{bg} B. Guan, \emph{Mean curvature motion of nonparametric hypersurfaces with contact angle
condition}, Elliptic and parabolic method in geometry, Peters, A.K.,
Wellesley (MA), 1996, pp. 47--56.

\bibitem{gh1} G. Huisken, \emph{Flow by mean curvature of convex surfaces into
spheres}, J. Differential Geom. {\bf 20} (1984) 237--266.

\bibitem{gh3} G. Huisken, \emph{Non-parametric mean curvature evolution with boundary
conditions}, J. Differential Equat. {\bf 77} (1989) 369--378.

\bibitem{gh2} G. Huisken, \emph{Asymptotic behavior for singularities of the mean curvature
flow}, J. Differential Geom. {\bf 31} (1990) 285--299.


\bibitem{ltu} P. Lions, N. Trudinger and J. Urbas, \emph{The Neumann problem for equations of Monge-Amp\`{e}re
type}, Commun. Pure Appl. Math. {\bf 39} (1986) 539--563.

\bibitem{ls} L. Shahriyari, \emph{Translating graphs by mean curvature
flow}, Geom. Dedicata {\bf 175} (2015) 57--64.

\bibitem{s} A. Stone, \emph{The mean curvature evolution of graphs},
Honour's Thesis, ANU 31, 1989.

\bibitem{ylx} Y.-L. Xin, \emph{Translating solitons of the mean curvature
flow}, Calc. Var. Partial Differential Equations {\bf 54} (2015)
1995--2016.


\bibitem{hyz} H.-Y. Zhou, \emph{Nonparametric mean curvature type flows of graphs with contact angle
conditions}, available online at arXiv:1702.02449v1.





\end{thebibliography}
\end{document}